\documentclass[11pt]{article}
\input{epsf.sty}
\usepackage{amsthm}
\usepackage{amsmath}
\usepackage{latexsym,bm}
\usepackage{floatflt}
\usepackage{fancyhdr}

\usepackage{amsfonts}
\usepackage{mathrsfs}
\usepackage{amssymb}
\usepackage{mathrsfs}
\usepackage{indentfirst}
\allowdisplaybreaks
\newtheorem{theorem}{\indent{Theorem}}[section]
\newtheorem{proposition}{\indent{Proposition}}[section]

\newtheorem{definition}{\indent{Definition}}[section]

\newtheorem{example}{\indent{example}}[section]

\begin{document}
 \pagenumbering{arabic}
\title
{ A kind of bifurcation of limit cycle from nilpotent critical point
    \thanks{This research was partially supported by the National
Natural Science Foundation of China (11371373,11201211). Corresponding
author: E-mail  address:liuyirong@163.com, lf0539@126.com}}
\author{{Liu Yirong $^{1}$ and Li Feng $^{2}$}\\
{(\small \it $^1$School of Mathematics and statistics, Central South University,
 Changsha, Hunan, 410083, P.R. China.})\\
{(\small \it $^2$School of Science, Linyi University, Linyi ,
Shandong, 276005, P.R. China.})}
\date{}
\maketitle \pagestyle{myheadings} \markboth{ Yirong  Liu and Feng Li
}{A kind of bifurcation of limit cycle from nilpotent critical point} \noindent

\begin{abstract}
In this paper, an interesting and new bifurcation phenomenon that limit cycles could be bifurcated from nilpotent node (focus) by changing its stability was investigated. It is different from lowing its  multiplicity in order to get limit cycles. We prove that $n^2+n-1$ limit cycles could be bifurcated by this way for $2n+1$ degree system. Moreover, this upper bound
could be reached. At last, we give two examples to show that $N(3)=1$ and $N(5)=5$.
\end{abstract}


\section{Introduction and Preliminary Knowledge}

One of the most intriguing aspects of the dynamics of real planar polynomial vector fields is the close relationship
between the center conditions and bifurcation of limit cycle.
Bifurcation of limit cycle from a high-order critical point in plane is becoming more and more important, there have been many results about this problem.
The Bogdanov Takens bifurcation from saddle-node point was discussed by Xiao and Zhan and De Maesschalck,
see \cite{Xiao-2007,Xiao-2008,De Maesschalck-2011} and \cite{Zhang-2004}, bifurcation of limit cycle from degenerate critical point was investigated by Han,
see \cite{Han-2012}. Especially, there were many results about bifurcations of limit cycles from nilpotent critical point, see \cite{Ly-1966,Ta-1974,Mo-1982,S-Z-2002,AG-2005} and \cite{AG-2006,Han-2008,Han-2013,Liu-Li-2009a,Liu-Li-2009b,Liu-Li-2011a,Liu-Li-2011b}.

The following planar real systems
\begin{equation}\label{e1.1}
    \frac{dx}{dt}=y+\sum\limits_{i+j=2}^{\infty}a_{ij}x^iy^j=\Phi(x,y),\ \
    \frac{dy}{dt}=\sum\limits_{i+j=2}^{\infty}b_{ij}x^iy^j=\Psi(x,y).
 \end{equation}
whose functions of right hand are analytic at the neighborhood of origin will be discussed in this paper.
The linear parts  of \eqref{e1.1} has double zero eigenvalues
but the matrix of the linearized system of \eqref{e1.1} at the
origin is not identically null. The origin $O(0,0)$ is called a
nilpotent singular point.

 \cite{Liu-1999} gave the definition of  the multiplicity of the point for
\begin{equation}\label{e1.4}
\begin{split}
    &\frac{dx}{dt}=\sum\limits_{k+j=0}^na_{kj}x^ky^j=P(x,y),\\
&\frac{dy}{dt}=\sum\limits_{k+j=0}^mb_{kj}x^ky^j=Q(x,y)
\end{split}
\end{equation}

\begin{definition}\label{d1.1}
Suppose $(x_0,y_0)$ is an isolate critical point \eqref{e1.4}( real or complex
), if the crossing number of $P(x,y)=0$ and $Q(x,y)=0$ at $(x_0,y_0)$ is N, then
the point $(x_0,y_0)$ is called a N-multiple singular point of \eqref{e1.5}, N
is called the multiplicity of the point $(x_0,y_0).$
\end{definition}

From Definition 2.1 in \cite{Liu-Li-2011a}, we have

\begin{proposition}\label{p1.1}
If \ $\Psi(x,y(x))=A x^N+o(x^N),\ \ A\neq0$, then the multiplicity
of the origin is a N-multiple singular point of \eqref{e1.1}.
\end{proposition}

A high order singular point could be broken into some low order singular point (real or complex) by a small parameters perturbation. Now, we consider the  perturbed system of
\eqref{e1.1} and \eqref{e1.4}
\begin{equation}\label{e1.5}
    \frac{dx}{dt}=\Phi(x,y)+h(x,y,\bm{\varepsilon}),\  \frac{dy}{dt}=\Psi(x,y)+g(x,y,\bm{\varepsilon}),
 \end{equation}
 and
\begin{equation}\label{e1.6}
    \frac{dx}{dt}=P(x,y)+h(x,y,\bm{\varepsilon}),\  \frac{dy}{dt}=Q(x,y)+g(x,y,\bm{\varepsilon}),
 \end{equation}
where
 $\bm{\varepsilon}=(\varepsilon_1,\varepsilon_2,\cdots,\varepsilon_l)$
is a finite dimension small parameters, $h(x,y,\bm{\varepsilon})$ and $
 g(x,y,\bm{\varepsilon})$ are power series of  $(x,y,\bm{\varepsilon})$
 with nonzero convergence radius, and $h(x,y,\bm{0})=
 g(x,y,\bm{0})=0$. From Theorem 1 in \cite{Liu-1999} and Theorem 2.1
 in \cite{Liu-Li-2011a}, it is easy to get the following theorem.

\begin{theorem}\label{t1.1}
Suppose the origin of system \eqref{e1.1} \emph{( }or \eqref{e1.4}\emph{)} is a
N-multiple singular point, then when $||\bm{\varepsilon}||<<1$
, the sum of multiplicity of all complex singular point in the sufficiently small neighborhood of origin of  \eqref{e1.5} \emph{(} or \eqref{e1.6}\emph{)}
is exactly  $N$.
\end{theorem}

\begin{example}\label{ex1.1}
From Proposition \ref{p1.1},  the multiplicity of the nilpotent origin of system
\begin{equation}\label{e1.7}
    \frac{dx}{dt}=y,\ \ \frac{dy}{dt}=Ax^N+yg(x,y)
\end{equation}
is exactly $N$,  $A\neq0,\ g(x,y)$ is analytic in the neighborhood of origin. When $||\bm{\varepsilon}||\ll1$, there are $m$  critical points $(\varepsilon_k,0)$ in the neighborhood of origin of system
\begin{equation}\label{e1.8}
    \frac{dx}{dt}=y,\ \ \frac{dy}{dt}=A\prod_{k=1}^m(x-\varepsilon_k)^{l_k}+yg(x,y),
\end{equation}
and their multiplicity are $l_k$, $k=1,2,\cdots,m$, where $l_1+l_2+\ldots+l_m=N$.
\end{example}

\begin{theorem}\label{t1.2}
Suppose the index of the origin of system \eqref{e1.1} \emph{( }or \eqref{e1.4}\emph{)} is
 $k$, then when $||\bm{\varepsilon}||<<1$,
the sum of index of all real singular point in the sufficient small neighborhood of origin of \eqref{e1.5} \emph{(} or \eqref{e1.6}\emph{)}is exactly $k$.
\end{theorem}

Liu etc gave the following definitions in order to compute Lyapunov constant in  \cite{Liu-Li-Huang-2008}.

\begin{definition}\label{d1.2}
Let $f_k,\ g_k$ be  polynomials with respect to
  $a_{ij}$$'s, b_{ij}$$'s,$ $k=1,2,\cdots$. If for an integer $m$, there exist
  polynomials with respect to  $a_{ij}$$'s,\ b_{ij}$$'s$:
$\xi_1^{(m)},$ $ \xi_2^{(m)},$ $ \cdots,$ $ \xi_{m-1}^{(m)}$, such
that
\begin{equation}\label{e1.9}
    f_m=g_m+\left(\xi_1^{(m)}f_1+\xi_2^{(m)}f_2+\cdots+\xi_{m-1}^{(m)}f_{m-1}\right).
\end{equation} Then, we say that
$f_m$ and $g_m$ is algebraic equivalent, written by $f_m\sim g_m$. If for any
integer $m$, we have $f_m\sim g_m$, we say that the sequences of
functions $\{f_m\}$ and $\{g_m\}$ are algebraic equivalent, written by
$\{f_m\}\sim \{g_m\}$.
\end{definition}

 The authors have proved that a nilpotent-node (nilpotent-focus) point with  multiplicity $2m+1$ could  be broken into a nilpotent-node (nilpotent-focus) with  multiplicity $2m-1$  and two complex singular points by a small parameters perturbation in \cite{Liu-Li-2011a}. If the stability at the element focus and nilpotent singular point is different, limit cycle will be bifurcated out from sufficiently small neighborhood of the element focus. In this paper, bifurcation of limit cycles from a nilpotent-node (nilpotent-focus) point will be investigated by changing the stability of the nilpotent-node (nilpotent-focus) point when the multiplicity is not decreased. It is different from \cite{Liu-Li-2011a}.

\section{Stability and bifurcation of limit cycle at nilpotent node (focus)}
 \setcounter{equation}{0}

Using theorem proved in \cite{Zhifen Zhang-1985}, we have

\begin{proposition}\label{p2.1}
Suppose that the function $y=y(x)$ satisfies $\Phi(x,y(x))=0,\ y(0)=0$, and
\begin{equation}\label{e2.1}
\begin{split}
  &\Psi(x,y(x))=\alpha_{2m+1} x^{2m+1}+o(x^{2m+1}), \ \ \ \ \alpha_{2m+1}<0,\\
&\left(\frac{\partial \Phi}{\partial x}+
 \frac{\partial \Psi}{\partial y}\right)_{y=y(x)}=\beta_{2n} x^{2n}+o(x^{2n}),\ \ \beta_{2n}\neq0,
\end{split}
\end{equation}
where $n,m$ are positive integers, then the origin of  \eqref{e1.1} is  a nilpotent-node (nilpotent-focus) point with  multiplicity $2m+1$, and the origin to be a nilpotent-node if and only if one of the following conditions is satisfied:
\begin{equation}\label{e2.2}
\begin{split}
  C_1:\ \ & 2n<m,\ \ \alpha_{2m+1}<0; \\
C_2:\ \ & 2n=m,\ \  \alpha_{2m+1}<0,\ \  \beta_{2n}^2+4(m+1)\alpha_{2m+1}\geqslant0.
\end{split}
\end{equation}
\end{proposition}
Furthermore,

\begin{theorem}\label{t2.1}
Suppose that the function $y=y(x)$ satisfies $\Phi(x,y(x))=0,\ y(0)=0$, and \eqref{e2.1} holds,  then multiplicity of the origin of system \eqref{e1.1} is $2m+1$ , Lyapunov constants are
\begin{equation}\label{e2.3}
    V_n=\beta_{2n},
\end{equation}
namely it is stable when $\beta_{2n}<0$ and unstable when $\beta_{2n}>0$.
\end{theorem}

\begin{proof}
From the discussions in \cite{Ta-1974} and \cite{AG-2006}, under conditions in theorem \ref{t2.1}, system \eqref{e1.1} could be transformed into Li\'{e}nard system
\begin{equation}\label{e2.5}
    \frac{du}{d\tau}=v,\ \ \frac{dv}{d\tau}=\alpha_{2m+1} u^{2m+1}+\beta_{2n}vu^{2n}g(u)
\end{equation}
by the following analytic changes
\begin{equation}\label{e2.4}
    \begin{split}
    & u=x+\sum_{k+j=2}^{\infty}a'_{kj}x^ky^j,\\
    & v=y+\sum_{k+j=2}^{\infty}b'_{kj}x^ky^j,\\
    &\frac{dt}{d\tau}=1+\sum_{k+j=1}^{\infty}c'_{kj}x^ky^j,
    \end{split}
\end{equation}
where $g(u)$ is analytic at $u=0$, and $g(0)=1$. Let $V=v^2-\alpha_{2m+1} u^{2m+2}$, Ôò
\begin{equation}\label{e2.6}
    \left.\frac{dV}{d\tau}\right|_{\eqref{e2.5}}=2\beta_{2n}v^2u^{2n}g(u),
\end{equation}
So the conclusion in Theorem \ref{t2.1} holds.
\end{proof}

The Theorem \ref{t2.1} leads to the following theorem
\begin{theorem}\label{t2.2}
Suppose that the function $y=y(x)$ satisfies $\Phi(x,y(x))=0,\ y(0)=0$, and
\begin{equation}\label{e2.7}
\begin{split}
  &\Psi(x,y(x))=\alpha_{2m+1} x^{2m+1}+o(x^{2m+1}), \ \ \ \ \alpha_{2m+1}<0,\\
&\left(\frac{\partial \Phi}{\partial x}+
 \frac{\partial \Psi}{\partial y}\right)_{y=y(x)}=\sum_{k=1}^n\beta_{2k} \left(x^{2k}+o(x^{2k})\right),\ \ \beta_{2n}\neq0,
\end{split}
\end{equation}
where $n,m$ are positive integers, then there exist $n-1$ limit cycles in the neighborhood of origin of system \eqref{e1.1} when
\begin{equation}\label{e2.8}
    0<|\beta_2|\ll|\beta_4|\ll\cdots|\beta_{2n}|,\ \ \beta_{2k }\beta_{2k+2 }<0,\ k=1,2,\cdots,n-1.
\end{equation}
\end{theorem}

\begin{example}\label{ex2.1}
From Theorem \ref{t2.2}, when \eqref{e2.8} holds, there exist $n-1$ limit cycles in the neighborhood of origin of system
\begin{equation}\label{e2.9}
    \frac{dx}{dt}=y,\ \ \frac{dy}{dt}=-x^{2m+1}+y\sum_{k=1}^n\beta_{2k} x^{2k}.
\end{equation}
\end{example}

Suppose $O$ is a nilpotent-node of system
\begin{equation}\label{e2.10}
 \frac{dx}{dt}=y+\sum_{k+j=2}^{2n+1}a_{kj}x^ky^j,\ \ \frac{dy}{dt}=\sum_{k+j=2}^{2n+1}b_{kj}x^ky^j,
\end{equation}
we denote the number of limit cycles which could be bifurcated from origin of  \eqref{e2.10} by changing the stability of the nilpotent-node point when the multiplicity is not decreased by $N(2n+1)$. It is easy to know that multiplicity of the nilpotent-node point  $O$ is no more than $(2n+1)^2$ from Bezout theorem and definition \ref{d1.1}. Combining with \ref{p2.1}, we could get

\begin{theorem}\label{t2.3}
\begin{equation}\label{e2.11}
    N(2n+1)\leq n^2+n-1.
\end{equation}
\end{theorem}

We will give two examples in Section 3 and Section 4 to show that the upper bound is arrival when $n=1,n=2 $ in \eqref{e2.11}, namely $N(3)=1,\ N(5)=5$.

\section{N(3)=1}
 \setcounter{equation}{0}

In this section, we will prove that $N(3)=1$. Considering the following cubic system
\begin{equation}\label{e3.1}
    \begin{split}
&\frac{dx}{dt}=y +x^2 + \varepsilon^2 y^2+ \varepsilon^2 x^2 y  - x y^2 + \varepsilon y^3=X(x,y),\\
&\frac{dy}{dt}=- 2 x y- 2 \varepsilon y^2-2 x^3  - 2 \varepsilon x^2 y  - 2 y^3=Y(x,y).
    \end{split}
\end{equation}
For system \eqref{e3.1},  a solution for $X(x,y(x))=0,\ \ y(0)=0$ is
\begin{equation}\label{e3.2}
    y=y(x)=-x^2 + x^5 + \varepsilon x^6 + \varepsilon^2 x^7 + (-2 + \varepsilon^3) x^8+o(x^8),
\end{equation}
and
\begin{equation}\label{e3.3}
\begin{split}
&Y(x,y(x))=-2x^9+o(x^9),\\
&\left.\left(\frac{\partial X}{\partial x}+\frac{\partial Y}{\partial y}\right)\right|_{y=y(x)}=2\varepsilon x^2
(1-\varepsilon x)-7x^4+o(x^4),
\end{split}
\end{equation}

From \eqref{e3.3} ,
 $\beta_{2}=2\varepsilon,\ \ \beta_{4}=-7,\ \ \alpha_9=-2<0$, then  $\Delta=\beta_{4}^2+20\alpha_9=9>0$ when $\varepsilon=0$,  theorem \ref{t2.2} shows that

\begin{theorem}\label{t3.1}
The origin of \eqref{e3.1} is  a nilpotent-node point of multiplicity 9, and there is a limit cycle in the neighborhood of origin of system \eqref{e3.1} when $0<\varepsilon\ll1$.
 \end{theorem}

\section{ $N(2)=5$}
 \setcounter{equation}{0}

In this section, we will prove that the upper bound could be reached when $n=2$. A class of $Z_2$ quintic system  with 25-multiple nilpotent node $O(0,0)$
\begin{equation}\label{e4.1}
    \begin{split}
&\frac{dx}{dt}=y+\sum_{k+j=3}a_{kj}x^ky^j+\sum_{k+j=5}a_{kj}x^ky^j=X(x,y),\\
&\frac{dy}{dt}=\sum_{k+j=3}b_{kj}x^ky^j+\sum_{k+j=5}b_{kj}x^ky^j=Y(x,y),
    \end{split}
\end{equation}
where
\begin{equation}\label{e4.2}
   \begin{split}
   a_{30} =& 1,\ \ a_{21} = 7 \lambda_1,\ \ a_{12} = \lambda_1 \lambda_3,\\
   a_{03} =& \frac{1}{8} (-1029 \lambda_1^3 + 140 \lambda_1^4 + 343 \lambda_1^2 \lambda_2 - 12 \lambda_1^3 \lambda_2 -
     35 \lambda_1 \lambda_2^2 + \lambda_2^3 \\
     - &28 \lambda_1^2 \lambda_3 + 4 \lambda_1 \lambda_2 \lambda_3 + 16 \lambda_4 -
     56 \lambda_1^2 \lambda_5 - 8 \lambda_1 \lambda_2 \lambda_5),\\
     a_{50} =& 0,\ \ a_{41} = \lambda_1 \lambda_5,\ \ a_{32} = \lambda_4,\\
     a_{23} =& \frac{1}{4} \lambda_1 (-343 \lambda_1^4 + 4 \lambda_1^5 + 70 \lambda_1^3 \lambda_2 - 3 \lambda_1^2 \lambda_2^2 -
     4 \lambda_1^3 \lambda_3\\
      +& 28 \lambda_4 - 196 \lambda_1^2 \lambda_5 + 8 \lambda_1^3 \lambda_5 + 4 \lambda_1 \lambda_3 \lambda_5 -
     4 \lambda_1 \lambda_5^2),\\
   a_{05} =& \frac{1}{16} (-50421 \lambda_1^7 + 1960 \lambda_1^8 - 16 \lambda_1^9 + 19894 \lambda_1^6 \lambda_2 -
     392 \lambda_1^7 \lambda_2 - 2744 \lambda_1^5 \lambda_2^2 + 16 \lambda_1^6 \lambda_2^2 + 154 \lambda_1^4 \lambda_2^3\\
      +&
     686 \lambda_1^2 \lambda_2 \lambda_4 - 24 \lambda_1^3 \lambda_2 \lambda_4 - 70 \lambda_1 \lambda_2^2 \lambda_4 + 2 \lambda_2^3 \lambda_4 -
     56 \lambda_1^2 \lambda_3 \lambda_4 + 8 \lambda_1 \lambda_2 \lambda_3 \lambda_4 + 16 \lambda_4^2 + 14406 \lambda_1^5 \lambda_5 \\
     +&
     1960 \lambda_1^6 \lambda_5 - 64 \lambda_1^7 \lambda_5 - 4802 \lambda_1^4 \lambda_2 \lambda_5 - 616 \lambda_1^5 \lambda_2 \lambda_5 +
     490 \lambda_1^3 \lambda_2^2 \lambda_5 + 32 \lambda_1^4 \lambda_2^2 \lambda_5 - 14 \lambda_1^2 \lambda_2^3 \lambda_5\\
      +&
     392 \lambda_1^4 \lambda_3 \lambda_5 + 64 \lambda_1^5 \lambda_3 \lambda_5 - 56 \lambda_1^3 \lambda_2 \lambda_3 \lambda_5 -
     112 \lambda_1^2 \lambda_4 \lambda_5 - 16 \lambda_1 \lambda_2 \lambda_4 \lambda_5 - 64 \lambda_1^5 \lambda_5^2 +
     112 \lambda_1^3 \lambda_2 \lambda_5^2),\\
      \end{split}
\end{equation}
\begin{equation}\label{e4.3}
   \begin{split}
      a_{14} =& \frac{1}{8} \lambda_1 (-7203 \lambda_1^5 + 84 \lambda_1^6 + 1813 \lambda_1^4 \lambda_2 - 4 \lambda_1^5 \lambda_2 -
     133 \lambda_1^3 \lambda_2^2 + 3 \lambda_1^2 \lambda_2^3 - 84 \lambda_1^4 \lambda_3 \\
     +& 4 \lambda_1^3 \lambda_2 \lambda_3 +
     8 \lambda_3 \lambda_4 - 1029 \lambda_1^3 \lambda_5 + 308 \lambda_1^4 \lambda_5 + 343 \lambda_1^2 \lambda_2 \lambda_5 -
     20 \lambda_1^3 \lambda_2 \lambda_5\\
      -& 35 \lambda_1 \lambda_2^2 \lambda_5 + \lambda_2^3 \lambda_5 - 84 \lambda_1^2 \lambda_3 \lambda_5
      +
     4 \lambda_1 \lambda_2 \lambda_3 \lambda_5 + 56 \lambda_1^2 \lambda_5^2 - 8 \lambda_1 \lambda_2 \lambda_5^2),\\
      b_{30} =& 0,\ \ b_{21} = \lambda_1,\ \ b_{12} = -\lambda_1 (7 \lambda_1 - \lambda_2),\\
    b_{03} = &\frac{1}{4} \lambda_1 (49 \lambda_1^2 + 4 \lambda_1^3 - 14 \lambda_1 \lambda_2 + \lambda_2^2),\\
    b_{50} =& \lambda_1,\ \ b_{41} = \lambda_1 \lambda_2,\\
   b_{32} =& \frac{1}{4} \lambda_1 (-147 \lambda_1^2 + 4 \lambda_1^3 + 14 \lambda_1 \lambda_2 + \lambda_2^2 + 4 \lambda_1 \lambda_3 -
     4 \lambda_1 \lambda_5),\\
     b_{23} =& \frac{1}{8} \lambda_1 (-343 \lambda_1^3 + 196 \lambda_1^4 + 147 \lambda_1^2 \lambda_2 - 12 \lambda_1^3 \lambda_2 -
     21 \lambda_1 \lambda_2^2 + \lambda_2^3\\
      - &84 \lambda_1^2 \lambda_3 + 12 \lambda_1 \lambda_2 \lambda_3 + 8 \lambda_4 +
     56 \lambda_1^2 \lambda_5 - 16 \lambda_1 \lambda_2 \lambda_5),\\
     b_{14} = &-\frac{1}{8}
     \lambda_1 (-7203 \lambda_1^4 + 294 \lambda_1^5 + 8 \lambda_1^6 + 3430 \lambda_1^3 \lambda_2 - 84 \lambda_1^4 \lambda_2 -
     588 \lambda_1^2 \lambda_2^2 + 6 \lambda_1^3 \lambda_2^2 \\
     +& 42 \lambda_1 \lambda_2^3 - \lambda_2^4
      - 294 \lambda_1^3 \lambda_3 -
     16 \lambda_1^4 \lambda_3 + 84 \lambda_1^2 \lambda_2 \lambda_3 - 6 \lambda_1 \lambda_2^2 \lambda_3\\
      +& 56 \lambda_1 \lambda_4 - 8 \lambda_2 \lambda_4 +
     98 \lambda_1^3 \lambda_5 + 24 \lambda_1^4 \lambda_5 - 84 \lambda_1^2 \lambda_2 \lambda_5 + 10 \lambda_1 \lambda_2^2 \lambda_5),\\
     b_{05} =&\frac{1}{32} \lambda_1 (-50421 \lambda_1^5 - 6860 \lambda_1^6 + 672 \lambda_1^7 + 31213 \lambda_1^4 \lambda_2 +
     2156 \lambda_1^5 \lambda_2 - 64 \lambda_1^6 \lambda_2 - 7546 \lambda_1^3 \lambda_2^2 \\
     - &196 \lambda_1^4 \lambda_2^2 +
     882 \lambda_1^2 \lambda_2^3 + 4 \lambda_1^3 \lambda_2^3 - 49 \lambda_1 \lambda_2^4 + \lambda_2^5 - 1372 \lambda_1^4 \lambda_3 -
     224 \lambda_1^5 \lambda_3 + 588 \lambda_1^3 \lambda_2 \lambda_3\\
      +& 32 \lambda_1^4 \lambda_2 \lambda_3 - 84 \lambda_1^2 \lambda_2^2 \lambda_3 +
     4 \lambda_1 \lambda_2^3 \lambda_3 + 392 \lambda_1^2 \lambda_4 + 32 \lambda_1^3 \lambda_4 - 112 \lambda_1 \lambda_2 \lambda_4 +
     8 \lambda_2^2 \lambda_4 \\
     +& 224 \lambda_1^5 \lambda_5 - 392 \lambda_1^3 \lambda_2 \lambda_5 - 64 \lambda_1^4 \lambda_2 \lambda_5 +
     112 \lambda_1^2 \lambda_2^2 \lambda_5 - 8 \lambda_1 \lambda_2^3 \lambda_5).
   \end{split}
\end{equation}
will be investigated in this section.

Suppose that $y=y(x)$ is the only solution of  $X(x,y(x))=0,\ \ y(0)=0$, $y(x)$ and  $Y(x,y(x))$ are odd functions of  $x$ because \eqref{e4.1} is $Z_2$ equivalent, and  $\left.\left(\frac{\partial X}{\partial x}+\frac{\partial Y}{\partial y}\right)\right|_{y=y(x)}$ is even function of  $x$. We have
\begin{equation}\label{e4.6}
    \begin{split}
&Y(x,y(x))=\alpha_{25}x^{25}+o(x^{25}),\\
&\left.\left(\frac{\partial X}{\partial x}+\frac{\partial Y}{\partial y}\right)\right|_{y=y(x)}=\sum_{k=1}^6\beta_{2k}x^{2k}+o(x^{12}),
    \end{split}
\end{equation}
where
\begin{equation}\label{e4.7}
    \begin{split}
&\alpha_{25}=-\frac{1}{16} \lambda_1^{10} (-343 \lambda_1^2 + 4 \lambda_1^3 + 70 \lambda_1 \lambda_2 - 3 \lambda_2^2 - 4 \lambda_1 \lambda_3 +  8 \lambda_1 \lambda_5)^2,\\
&\beta_2=3 + \lambda_1,\ \ \beta_4 \sim 3 (56 + \lambda_2),\\
&\beta_6 \sim -\frac{3}{4} (-59 + 28 \lambda_3 - 40 \lambda_5),\\
&\beta_8 \sim -6 (675 + \lambda_4 - 93 \lambda_5),\\
&\beta_{10} \sim -\frac{27}{49} (477 + 4 \lambda_5) (93 + \lambda_5),\\
&\beta_{12} \sim 972 (477 + 4 \lambda_5).
    \end{split}
\end{equation}

\begin{theorem}\label{t4.1}
If
\begin{equation}\label{e4.8}
\begin{split}
    &\lambda_1=-3-\varepsilon_1,\ \ \lambda_2=-56+\varepsilon_2,\ \ \lambda_3=\frac{1}{4} (-523 + 4 \varepsilon_3 + 40 \varepsilon_5),\\
    &\lambda_4=-9324 - \varepsilon_4 + 651 \varepsilon_5,\ \ \lambda_5=-93 + 7 \varepsilon_5,
    \end{split}
\end{equation}
then the origin of system \eqref{e4.1} is  a nilpotent-node point with multiplicity 25, when
 \begin{equation}\label{e4.9}
 0<\varepsilon_1\ll\varepsilon_2\ll\varepsilon_3\ll\varepsilon_4\ll\varepsilon_5\ll1,
\end{equation}
 there exist 5 limit cycles in the neighborhood of system \eqref{e4.1}.
\end{theorem}
\begin{proof}
From \eqref{e4.7}, $\beta_{25}<0$ when \eqref{e4.8} and \eqref{e4.8} hold, and
\begin{equation}\label{e4.9}
    \begin{split}
&\beta_2=-\varepsilon_1,\ \ \beta_4 \sim 3 \varepsilon_2,\ \ \beta_6 \sim -21\varepsilon_3 ,\  \ \beta_8 \sim 6\varepsilon_4,\\
 &\beta_{10} \sim -405 \varepsilon_5+o(\varepsilon_5),\ \ \beta_{12} \sim 102060,
    \end{split}
\end{equation}
and when $\varepsilon_1=\varepsilon_2=\varepsilon_3=\varepsilon_4=\varepsilon_5=0$, we have
\begin{equation}\label{e4.10}
    \Delta=\beta_{12}^2+52\alpha_{25}=4198383900>0.
\end{equation}
So the conclusion in  \ref{t4.1} hold from theorem \ref{t2.2}.
\end{proof}


\begin{thebibliography}{99}

\bibitem[Xiao D., 2008]{Xiao-2008} Xiao D., Bifurcations on a five-parameter family of planar vector field. J. Dyn. and Diff. Equa, 2008, \textbf{20}(4): 961-980.

 \bibitem[Xiao D., 2007]{Xiao-2007}Xiao D, Ruan S., Global analysis of an epidemic model with nonmonotone incidence rate. Math. Bio, 2007, \textbf{208}(2): 419-429.

\bibitem[Tang Y, Zhang W., 2004]{Zhang-2004} Tang Y, Zhang W. , Bogdanov-Takens bifurcation of a polynomialdifferential system in biochemical reaction. Comp. and Math. with App., 2004, \textbf{48}(5): 869-883.

\bibitem[Han M, Yu P., 2012]{Han-2012}Han M, Yu P. Normal forms, melnikov functions and bifurcations of limit cycles. Springer, \ 2012.

\bibitem[De Maesschalck P, Dumortier F., 2011]{De Maesschalck-2011} De Maesschalck P, Dumortier F. Slow¨Cfast Bogdanov¨CTakens bifurcations. Journal of Differential Equations, 2011, \textbf{250}(2): 1000-1025.

 \bibitem[Han M, Jiang J, Zhu H., 2008]{Han-2008} Han M, Jiang J, Zhu H. Limit cycle bifurcations in near-Hamiltonian systems by perturbing a nilpotent center. Int. Int. J. Bifur. Chaos., 2008, \textbf{18}(10): 3013-3027.

 \bibitem[Han M, Romanovski V G., 2012]{Han-2013} Han M, Romanovski V G. Limit cycle bifurcations from a nilpotent focus or center of planar systems. Abstract and Applied Analysis. Hindawi Publishing Corporation,  2012.

\bibitem[Ame\'{l}ikin etc, 1982]{Am-1982} Amel\'{}ikin, B. B. Lukashivich, H. A. and Sadovski, A. P.
 Nonlinear Oscillations in Second Order Systems, Minsk, BGY
lenin.B. I. Press.\ 1982,  (in Russian).

 \bibitem[\'{A}lvarez and  Gasull, 2005]{AG-2005} \'{A}lvarez,M.J.,\ Gasull,A. "Monodromy and stablility
 for nipotent critical points", IJBC. Vol.15. \textbf{4}, 1253-1265.

 \bibitem[\'{A}lvarez and Gasull, 2006]{AG-2006}\'{A}lvarez, M.J.,\ Gasull,A. Cenerating limits
 cycles from a nipotent critical point via normal forms, J.Math.
 Anal.Appl,2006, \textbf{318},271-287.

  \bibitem[Liu Y.R., 1999]{Liu-1999} Liu Yirong, Multiplicity of higher order singular point
   of differential autonomous system.
  J. Cent. South Univ. Techonol., 1999,\textbf{30}(3): 622-623.

  \bibitem[Liu-Li-Huang, 2008]{Liu-Li-Huang-2008} Liu Y., Li J., Huang W., Singular Point Values, Center Problem and Bifurcations of Limit Cycles of Two Dimensions Differential Autonomous Systems, Science Press, Beijing, China, 9-13.

   \bibitem[Liu-Li, 2009a]{Liu-Li-2009a} Liu Y. , Li J., New study on  the center
  problem and bifurcations of limit cycles
for The Lyapunov system  (I), Int. J. Bifur. Chaos. 2009,
\textbf{19}(11), 3791-3801.

\bibitem[Liu-Li, 2009b]{Liu-Li-2009b} Liu Y. ,Li J., New study on  the center
  problem and bifurcations of limit cycles for The Lyapunov system  (II), Int. J. Bifur. Chaos. 2009,
\textbf{19}(9), 3087-3099.

\bibitem[Liu-Li, 2011a]{Liu-Li-2011a} Liu Y. , Li J., Bifurcations of
limit cycles created by a multiple nilponent critical point of
 planar dynamical systems, Int. J. Bifur. Chaos. 2011, \textbf{21}(2), 497-504.

\bibitem[Liu-Li, 2011b]{Liu-Li-2011b} Liu Y. , Li J.,  On the study of three-order
 nilpotent critical points: Integral factor method, Int. J. Bifur. Chaos. 2011,\textbf{21}(5),
 1293-1309.

\bibitem[Lyapunov,A.M.,1966]{Ly-1966}Lyapunov,A.M.,[1966]Stability of Motion,
Mathematics in Science and Engineering, Vol.30 (Academic
Press,NY-London).

\bibitem[Moussu,1982]{Mo-1982} Moussu.R, "Sym\'{e}trie et forme
normale des centres et foyers d\'{e}g\'{e}n\'{e}r\'{e}s", Ergodic
Theory Dynam. Systems \textbf{2}, 241-251.


\bibitem[Strozyna and Zoladek,2002]{S-Z-2002}Strozyna.E.,\
Zoladek.H., "The analytic normal for the nilpotent singularty",
J.Differential Equations,\ \textbf{179},479-537.

\bibitem[Takens,1974]{Ta-1974}  Takens.F, "Singularties of vector
 fields",\ Inst.Hautes \'{E}tudes Sci. Publ.Math., \textbf{43},47-100.

\bibitem[Zhifen Zhang-1985]{Zhifen Zhang-1985} Zhifen Z, "Qualitative Theory of Ordinary Differential Equations", Science Press, 1985.

\end{thebibliography}
\end{document}